\documentclass[12pt,reqno,a4paper]{amsart}

\usepackage{a4wide}
\usepackage{dsfont} 	
\usepackage{amssymb}
\usepackage{mathrsfs}
\usepackage{upgreek}

\newtheorem{proposition}{Proposition}[section]
  \newtheorem{theorem}[proposition]{Theorem}
  \newtheorem{lemma}[proposition]{Lemma}
  \newtheorem{corollary}[proposition]{Corollary}
\theoremstyle{definition}
  \newtheorem{definition}[proposition]{Definition}

\newcommand{\cst}{\ifmmode\mathrm{C}^*\else{$\mathrm{C}^*$}\fi}

\newcommand{\qqquad}{\quad\qquad}
\newcommand{\ov}{\overline}

\newcommand{\NN}{\mathbb{N}}
\newcommand{\RR}{\mathbb{R}}
\newcommand{\CC}{\mathbb{C}}
\newcommand{\GG}{\mathbb{G}}

\newcommand{\cc}{\text{\rm\tiny{c}}}

\DeclareMathOperator{\Irr}{Irr}
\DeclareMathOperator{\Mor}{Mor}
\DeclareMathOperator{\End}{End}
\DeclareMathOperator{\Tr}{Tr}
\DeclareMathOperator{\uRep}{uRep}

\newcommand{\vecRho}{\overset{{}_{\scriptscriptstyle\rightarrow}}{\uprho}}
\newcommand{\ivecRho}{\overset{{}_{\scriptscriptstyle\leftarrow}}{\uprho}}
\newcommand{\tp}{\!\!\!
{\scriptstyle
\text{
\raisebox{0.8pt}{
\textcircled{\raisebox{-1.7pt}{$\top$}}
} 
} 
} 
\!\!\!}
\newcommand{\stp}{\!\!\!\!\!\:
{\scriptscriptstyle
\text{
\raisebox{0.5pt}{
\textcircled{\raisebox{-1.4pt}{$\top$}}
} 
} 
} 
\!\!\!\!\!\:}

\numberwithin{equation}{section}

\DeclareMathAlphabet{\mathpzc}{OT1}{pzc}{m}{it}

\begin{document}

\baselineskip=17pt

\author[J. Krajczok]{Jacek Krajczok}
\address{Department of Mathematical Methods in Physics, Faculty of Physics, University of Warsaw, Poland}
\email{jk347906@okwf.fuw.edu.pl}

\begin{abstract}
\end{abstract}

\title[Symmetry of eigenvalues]{Symmetry of eigenvalues of operators associated with representations of compact quantum groups}




\begin{abstract}
We ask the question whether for a given unitary representation $U$ of a compact quantum group $\GG$ the associated operator $\uprho_{U}\in\Mor(U,U^{\cc\,\cc})$ has spectrum invariant under inversion -- in this case we say that $\uprho_{U}$ has symmetric eigenvalues. This does not have to be the case. However, we give affirmative answer whenever a certain condition on the growth of dimensions of irreducible subrepresentations of tensor powers of $U$ is imposed. This condition is satisfied whenever $\widehat{\GG}$ is of subexponential growth.
\end{abstract}

\maketitle

\section{Introduction}
Let $\GG$ be a compact quantum group. We say that $\GG$ is of Kac type if for every finite dimensional unitary representation $U$ the contragradient representation $U^{\cc}$ is unitary (\cite[Definition 1.3.8.]{NT}). There are various equivalent formulations of this property related to the Haar integral on $\GG$, Haar integrals on $\widehat{\GG}$ (see \cite{qLor}) or the antipode of $\GG$ -- see e.g. \cite{introtocqgdqg, NT, cqg}. It is known that not all compact quantum groups are of Kac type -- counterexamples are given e.g. by $SU_q(2)$ groups when $q\notin\{-1,1\}$ (\cite{su2}) or (some of) the free unitary and orthogonal groups (\cite{uqg}). Even if $U^{\cc}$ is not unitary, it is still a representation and we can form a second contragradient $U^{\cc\,\cc}$. It turns out that $U^{\cc\,\cc}$, as in the classical case, is equivalent to $U$. This equivalence is given by a positive operator $\uprho_{U}\in\Mor(U,U^{\cc\,\cc})$ which is characterized uniquely by the property $\Tr(T\, \uprho_{U})=\Tr(T\, {\uprho_{U}}^{-1})\;(T\in\End(U))$ (\cite[Section 3, 4]{NT}). In this paper we consider a problem when the spectrum of $\uprho_{U}$ coincides with the spectrum of ${\uprho_{U}}^{-1}$ (counting with multiplicities) -- this will be stated more precisely in Section \ref{mainSect}. There are examples of representations which does not have this property, however, we present some sufficient conditions for this to hold. Proposition \ref{simpleSym} states that this symmetry condition holds when $U$ and its conjugate representation are equivalent, Theorem \ref{orderF} relates this property to the rate of growth of dimensions of irreducible subrepresentations of tensor powers of $U$. Moreover, Corollary \ref{cor} states that $\uprho_{U}$ has this property whenever $\widehat{\GG}$ is of subexponential growth. Proof of Theorem \ref{orderF} uses functions $d_t$  $(t\in\RR)$ which are generalizations of classical and quantum dimension of a unitary representation and will be introduced in the next section.

We refer to \cite[Chapter 1]{NT} for necessary definitions and basic theory of compact quantum groups. Furthermore, most of our notational conventions agree with those of \cite{NT}.

\section{Notation}\label{notation}

Throughout the paper $\GG$ will stand for a fixed compact quantum group. We will denote by $\uRep(\GG)$ the class of finite dimensional, unitary representations of $\GG$ and by $\Irr(\GG)$ the set of equivalence classes of irreducible unitary representations. For each class $\alpha\in\Irr(\GG)$ we choose a representative $U^\alpha\in\uRep(\GG)$. We will write $\alpha$ instead of $U^{\alpha}$ in objects which depends only on equivalence class -- e.g. $\dim(\alpha)=\dim(U^{\alpha})$. Let $U,V\in\uRep(\GG)$ be two representations. Their tensor product will be denoted by $U\tp V$. By \cite[Proposition 1.4.4.]{NT} there exists a unique positive element of $\Mor(U,U^{\cc\, \cc})$, $\uprho_{U}$ such that
\begin{equation}\label{rhodef}
\Tr(\cdot\, \uprho_{U})=\Tr(\cdot\, {\uprho_{U}}^{-1}) \qqquad \textnormal{on } \End(U)=\Mor(U,U).
\end{equation}

For any real number $t\in\RR$ we define
\[
d_{t}(U)=\Tr({\uprho_{U}}^{t})\in]0,+\infty[
\]
(we know that $d_{t}(U)$ is a positive number because operator $\uprho_{U}$ is positive and invertible -- its spectrum lies in $]0,+\infty[$). Note that since operators $\uprho_{U}$ satisfy (\cite[Proposition 1.4.7., Theorem 1.4.9.]{NT})
\[
\uprho_{U\oplus V}=\uprho_{U}\oplus\uprho_{V},\quad
\uprho_{U\stp V}=\uprho_{U}\otimes\uprho_{V},\quad
\uprho_{\ov{U}}=({\uprho_{U}}^{-1})^{\top},
\]
we have 
\[
d_{t}(U\oplus V)=d_{t}(U)+d_{t}(V),\quad
d_{t}(U\tp V)=d_{t}(U)d_{t}(V),\quad
d_{t}(\ov{U})=d_{-t}(U),
\]
where $\ov{U}$ is the conjugate of $U$ (\cite[Definition 1.4.5.]{NT}). Moreover, $d_{t}(U)$ depends only on the equivalence class of $U$. It follows directly from the definition that $d_0(U)$ equals the dimension of $U$ and $d_{1}(U)=d_{-1}(U)$ is the quantum dimension of $U$ (\cite[Definition 1.4.1.]{NT}). In general, behaviour of the function $t\mapsto d_{t}(U)$ allows us to infer information about the structure of eigenvalues of $\uprho_{U}$.

Let us also define $\vecRho_{U}$ to be the list of eigenvalues of $\uprho_{U}$ (counting with multiplicities) in descending order and $\ivecRho_{U}$ to be the list of eigenvalues of ${\uprho_{U}}^{-1}$ (counting with multiplicities) in the descending order. We will treat $\vecRho_U$ and $\ivecRho_U$ as elements of $\RR^{\dim{U}}$. Observe that for $t\ge 1$ we have
\[
d_{t}(U)=(\|\vecRho_{U}\|_{t})^{t},
\]
where $\|\cdot\|_{t}$ is the $\ell_{t}$--norm on $\RR^{\dim{U}}$.

\section{Symmetry of eigenvalues}\label{mainSect}

Let $U\in\uRep(\GG)$. By \eqref{rhodef} we have
\begin{equation}\label{inv}
\|\vecRho_{U}\|_{1}=\Tr(\uprho_{U})=\Tr({\uprho_{U}}^{-1})=\|\ivecRho_{U}\|_{1}.
\end{equation}
In this section we would like to address the question, when the following stronger property holds
\begin{equation}\label{symmetry}
\vecRho_U=\ivecRho_U.
\end{equation}
If \eqref{symmetry} is true, we say that $\uprho_{U}$ has \emph{symmetric eigenvalues}.
If $\dim{U}\in\{1,2\}$ this is of course true, however if dimension of $U$ is larger, this need not be the case. Indeed, there exists compact quantum groups with fundamental representations which do not satisfy \eqref{symmetry}. For example, the construction of free quantum unitary groups begins with the choice of an invertible scalar matrix $F$ which satisfies $\Tr(F^*F)=\Tr((F^*F)^{-1})$. Then the operator $\uprho_U$ for the fundamental representation is $(F^*F)^\top$ (\cite[Example 1.4.2.]{NT}), hence one can easily choose matrix $F$ so that \eqref{symmetry} fails (e.g. take $F=\operatorname{diag}(y,x,x)\in \operatorname{GL}_{3}(\CC)$, where $y>1$ and $x>0$ is the postive solution of the equation $\frac{2}{\mathfrak{X}^{2}}+\frac{1}{y^2}=2\mathfrak{X}^2+y^2$).

Nevertheless there are cases where one can prove \eqref{symmetry}. We start with a simple one.

\begin{proposition}\label{simpleSym}
Let $U\in \uRep(\GG)$ be a unitary representation of $\GG$. If the representations $U$ and $\overline{U}$ are equivalent then \eqref{symmetry} holds.
\end{proposition}

\begin{proof}
This is an immediate consequence of \cite[Proposition 1.4.7.]{NT} which says that $\uprho_{\ov{U}}$ is the transpose of ${\uprho_U}^{-1}$. 
\end{proof}

\begin{corollary}
Equation \eqref{symmetry} holds for every finite dimensional unitary representation $U$ of the free orthogonal quantum group $A_{o}(F)$. In particular, it holds for every $U\in\uRep(SU_{q}(2))$.
\end{corollary}

\begin{proof}
Assume that $F\in\operatorname{GL}_{n}(\CC)$. Since every finite dimensional unitary representation is unitarily equivalent to a direct sum of irreducible ones, it is enough to assume that $U$ is irreducible. Irreducible representations of $A_{o}(F)$ are labeled (up to equivalence) by natural numbers: for every $r\in\NN$ there is an irreducible unitary representation $V^{(r)}\in\uRep(A_{o}(F))$ of dimension $z_{r}=\frac{x^{r+1}-y^{r+1}}{x-y}$, where $x$ and $y$ are solutions of the equation $\mathfrak{X}^{2}-n\mathfrak{X}+1=0$. Moreover, every irreducible representation of $A_{o}(F)$ is equivalent to $V^{(r)}$ for some $r\in \NN$ (\cite[Remark 6.4.11., Corollary 6.4.12.]{Timmermann}). Since $z_{r}\neq z_{r'}$ when $r\neq r'$, it follows that for each $r\in\NN$ the representation $\ov{V^{(r)}}$ is equivalent to $V^{(r)}$. In particular, $\ov U$ is equivalent to $U$. Second claim follows, because $SU_{q}(2)=A_{o}(F)$ for $F=
\begin{bmatrix} 
0 & 1 \\
-q^{-1} & 0
\end{bmatrix}
$ (\cite[Proposition 6.4.8.]{Timmermann}).
\end{proof}

In order to state our main theorem we have to introduce a function $\NN\ni n\mapsto P_{U}(n)\in\RR$ which tells us how the maximal dimension of irreducible subrepresentations of $n$-th tensor power of $U$ grows as we increase $n$.
\begin{definition}
For $U\in\uRep(\GG)$ and $n\in\NN$ put 
\[
P_{U}(n)=\max\{\dim \alpha_{1}, \dotsc,\dim \alpha_{N}\},
\]
where $\alpha_{1},\dotsc,\alpha_{N}\in\Irr(\GG)$ are such that $U^{\stp n}$ is equivalent to the direct sum $\bigoplus\limits_{k=1}^{N} U^{\alpha_{k}}$.
\end{definition}

The next theorem says that for any $U\in\uRep(\GG)$ eigenvalues of $\uprho_{U}$ are symmetric (i.e. \eqref{symmetry} holds) if $\NN\ni n\mapsto P_{U}(n)\in\RR$ grows subexponentially.

\begin{theorem}\label{orderF}
Let $U\in\uRep(\GG)$ and assume that
\begin{equation}\label{growth}
\forall\, c>1 \qqquad\lim_{n\to\infty}\frac{P_{U}(n)}{c^n}=0.
\end{equation}
Then eigenvalues of $\rho_{U}$ are symmetric, that is $\vecRho_U=\ivecRho_U$.
\end{theorem}

In the proof we will use standard inequalities for $\ell_p$--norms on $\RR^n$:
\[
\|x\|_{p'}\leq\|x\|_p\leq{n^{\frac{1}{p}-\frac{1}{p'}}}\|x\|_{p'},\qqquad{x}\in\RR^n,\:1\leq{p}\leq{p'}<+\infty
\]
which are a consequence of H\"older inequality, and the following elementary lemma:

\begin{lemma}\label{an1}
Let $a_1\geq{a_2}\geq\dotsm\geq{a_n}>0$ and $b_1\geq{b_2}\geq\dotsm\geq{b_m}>0$ be such that
\[
\sum_{i=1}^na_i^t=\sum_{j=1}^mb_j^t
\]
for every $t>1$. Then $n=m$ and $a_i=b_i$ for $i\in\{1,\dotsc,n\}$.
\end{lemma}

\begin{proof}[Proof of Theorem \ref{orderF}]
Let $n\in\NN$ and $\alpha_{1},\dotsc,\alpha_{N}\in\Irr(\GG)$ be such that
\[
U^{\stp{n}}\textnormal{ and }\bigoplus_{k=1}^N U^{\alpha_k}\textnormal{ are unitarily equivalent}
\]
(we are not assuming the $U^{\alpha_k}$'s are pairwise non-equivalent). For each $k\in\{1,\dotsc,N\}$ and $t> 1$ we have
\[
\begin{split}
d_t(\alpha_k)&=\bigl(\|\vecRho_{\alpha_k}\|_t\bigr)^t\leq\bigl(\|\vecRho_{\alpha_k}\|_1\bigr)^t=\bigl(\|\ivecRho_{\alpha_k}\|_1\bigr)^t\\
&\leq\bigl((\dim \alpha_k)^{1-\frac{1}{t}}\|\ivecRho_{\alpha_k}\|_t\bigr)^t=(\dim \alpha_k)^{t-1}\bigl(\|\ivecRho_{\alpha_k}\|_t\bigr)^t\leq{P_{U}(n)^{t-1}}d_{-t}(\alpha_k),
\end{split}
\]
where in the third step we used \eqref{inv}.

Similarly for any $k\in\{1,\dotsc,N\}$ and $t> 1$
\[
\begin{split}
d_{-t}(\alpha_k)&=\bigl(\|\ivecRho_{\alpha_k}\|_t\bigr)^t\leq\bigl(\|\ivecRho_{\alpha_k}\|_1\bigr)^t=
\bigl(\|\vecRho_{\alpha_k}\|_1\bigr)^t\\
&\leq\bigl((\dim \alpha_k)^{1-\frac{1}{t}}\|\vecRho_{\alpha_k}\|_t\bigr)^t=(\dim \alpha_k)^{t-1}\bigl(\|\vecRho_{\alpha_k}\|_t\bigr)^t\leq{P_{U}(n)}^{t-1}d_t(\alpha_k).
\end{split}
\]
Summing these inequalities over $i$ yields
\[
\begin{split}
d_t(U)^n=d_t\bigl(U^{\stp{n}}\bigr)&=\sum_{k=1}^{N}d_{t}(\alpha_{k})
\leq\sum_{k=1}^{N}P_{U}(n)^{t-1}d_{-t}(\alpha_{k})\\&=
P_{U}(n)^{t-1}d_{-t}(U^{\stp{n}})=
P_{U}(n)^{t-1}d_{-t}(U)^n,\\
d_{-t}(U)^n=d_{-t}\bigl(U^{\stp{n}}\bigr)&=
\sum_{k=1}^{N}d_{-t}(\alpha_k)
\leq\sum_{k=1}^{N}P_{U}(n)^{t-1}d_{t}(\alpha_{k})\\&=
P_{U}(n)^{t-1}d_{t}(U^{\stp n})=
{P_{U}(n)^{t-1}}d_t(U)^n.
\end{split}
\]
Thus
\[
1\leq\left(\frac{P_{\,U}(n)}{\bigl(d_t(U)/d_{-t}(U)\bigr)^{\frac{n}{t-1}}}\right)^{t-1}
\]
and
\[
1\leq\left(\frac{P_{\,U}(n)}{\bigl(d_{-t}(U)/d_t(U)\bigr)^{\frac{n}{t-1}}}\right)^{t-1}
\]
for all $t>1$.

Now, if $d_t(U)\neq{d_{-t}(U)}$ for some $t>1$, then either $\bigl(d_t(U)/d_{-t}(U)\bigr)^\frac{1}{t-1}$ or $\bigl(d_{-t}(U)/d_t(U)\bigr)^\frac{1}{t-1}$ is strictly greater then $1$. Setting $c$ to be this number we get
\[
1\leq\left(\frac{P_{\,U}(n)}{c^n}\right)^{t-1}.
\]
Now taking limit $n\to\infty$ gives a contradiction. It follows that we must have $d_t(U)=d_{-t}(U)$ for all $t>1$ and consequently $\vecRho_U=\ivecRho_U$ by Lemma \ref{an1}.
\end{proof}

For any unitary representation $U\in \uRep(\GG)$ and $n\in\NN$ let us define
\[
b(U,n)=\sum (\dim \alpha)^{2},
\]
where $\alpha$ ranges over all (nonequivalent) irreducible subrepresentations of $\bigoplus\limits_{k=0}^{n}U^{\stp k}$. It is natural to say that the dual $\widehat{\GG}$ has \emph{subexponential growth} if
\begin{equation}\label{subexp}
\lim_{n\to\infty}b(U,n)^{1/n}=1
\end{equation}
holds for every $U\in\uRep(\GG)$ (cf. \cite[Section 3]{growthD'Andrea} and \cite{growthBanica}).

\begin{corollary}\label{cor}
Assume that $\widehat{\GG}$ has subexponential growth. Then for every $U\in\uRep(\GG)$ eigenvalues of $\uprho_{U}$ are symmetric i.e. $\vecRho_{U}=\ivecRho_{U}$.
\end{corollary}

\begin{proof}
We have $P_{U}(n)\le b(U,n)$ for every $n\in \NN$, hence the equality \eqref{subexp} implies that condition \eqref{growth} from Theorem \ref{orderF} holds.
\end{proof}

\subsection*{Acknowledgements}
The author wish to thank Piotr So\l tan for guidance during his work on master thesis as well as for help in organizing the above material.

\bibliography{sym}{}
\bibliographystyle{plain}

\end{document}